\documentclass[11pt]{amsart}
\usepackage{amsmath, amsthm, amssymb}
\usepackage{graphicx}
\usepackage{tkz-euclide}
\usepackage{mathbbol}
\usepackage{txfonts}
\usepackage{color}

\title[The Null distance encodes causality]{The Null distance encodes causality}

\author[Sakovich]{A Sakovich}
\author[Sormani]{C Sormani}

\thanks{The authors began this research project while in residence at the Mathematical Sciences Research Institute (MSRI) funded by NSF Grant No. 0932078000. The research was also funded in part by Sormani's NSF grant DMS-1612409 and PSC-CUNY funding.  Sakovich was partly funded by the Swedish Research Council dnr. 2016-04511}

\newtheorem{thm}{Theorem}[section]
\newcommand{\bt}{\begin{thm}}
\newcommand{\et}{\end{thm}}

\newtheorem{cor}[thm]{Corollary}   

\newcommand{\bc}{\begin{cor}}
\newcommand{\ec}{\end{cor}}

\newtheorem{lem}[thm]{Lemma}   

\newcommand{\bl}{\begin{lem}}
\newcommand{\el}{\end{lem}}

\newtheorem{prop}[thm]{Proposition}
\newcommand{\bp}{\begin{prop}}
\newcommand{\ep}{\end{prop}}

\newtheorem{defn}[thm]{Definition}

\newcommand{\ben}{\begin{itemize}}
\newcommand{\een}{\end{itemize}}

\newcommand{\bd}{\begin{defn}}    
\newcommand{\ed}{\end{defn}}

\newtheorem{rmrk}[thm]{Remark}   

\newcommand{\br}{\begin{rmrk}}
\newcommand{\er}{\end{rmrk}}

\newtheorem{example}[thm]{Example}

\theoremstyle{definition}

\newcommand{\dhat}{\hat{d}}

\newcommand{\be}{\begin{equation}}

\newcommand{\ee}{\end{equation}}

\def\iff{\Longleftrightarrow}
\def\implies{\Longrightarrow}

\begin{document}

\maketitle

\begin{abstract}  
A Lorentzian manifold endowed with a time function, $\tau$, can be converted into a metric space using the null distance, $\hat{d}_\tau$, defined by Sormani and Vega.   We show that if the time function is a proper regular cosmological time function as studied by Andersson, Galloway and Howard, and also by Wald and Yip, or if, more generally, it satisfies the anti-Lipschitz condition of Chru\' sciel, Grant and Minguzzi, then the causal structure is encoded by the null distance in the following sense:
$$
\hat{d}_\tau(p,q)=\tau(q)-\tau(p) \iff q \textrm{ lies in the causal future of } p.
$$
As a consequence, 
in dimension $n+1$, $n\ge 2$, we prove
that if there is a bijective map between two such spacetimes,
$F: M_1\to M_2$, which preserves the cosmological time function, $  \tau_2(F(p))= \tau_1(p)$ for any $ p \in M_1$, 
and preserves the null distance,
$\hat{d}_{\tau_2}(F(p),F(q))=\hat{d}_{\tau_1}(p,q)$ for any $p,q\in M_1$,
then there is a Lorentzian isometry between them, $F_*g_1=g_2$.
This yields a canonical procedure allowing us to convert such spacetimes into unique metric spaces with causal structures and time functions.   This will be applied in our upcoming work to define Spacetime Intrinsic Flat Convergence.  \end{abstract}

\newpage

\section{Introduction}\label{sec:intro}

This paper is part of a series of papers developing the notion of  Spacetime Intrinsic Flat (SIF) Convergence of Lorentzian Manifolds as suggested by Shing-Tung Yau.   The overarching plan is to convert the Lorentzian manifolds canonically into unique metric spaces and then to take the intrinsic flat limit of these metric spaces \cite{Sormani-Oberwolfach-18}.   
One method of converting 
a Lorentzian manifold $(N,g)$ with a time function, $\tau$, into a metric space is to use the null distance, $\hat{d}_\tau$, developed by Sormani and Vega in \cite{SV-Null}.   This conversion process,
\be \label{conversion-map}
(N,g) \mapsto (N, \hat{d}_\tau, \tau)
\ee
is canonical for a Lorentzian manifold endowed with a regular cosmological time function, $\tau=\tau_{AGH}$, as defined by Andersson, Galloway, and Howard in \cite{AGH}.   

In this paper we prove that the conversion map in (\ref{conversion-map}) is one-to-one from isometry classes of Lorentzian manifolds to time preserving isometry classes of the metric spaces (see Theorem~\ref{Lorentzian-isom}).    In addition, we prove that the causal structure of $N$ is locally encoded by $(N, \hat{d}_\tau, \tau)$ (see Theorem~\ref{encodes-causality-locally}) and is globally encoded for spacetimes $N$ where $\tau$ is also a proper function  (see Theorem~\ref{proper-causality} within).  

Given a Lorentzian manifold, $(N,g)$, Andersson, Galloway and Howard \cite{AGH} have defined the notion of a canonical time function: 
\be\label{tau-AGH}
\tau_{AGH}(p)= \sup \{L_g(C)\,: \, \textrm{ future timelike }C:[0,1]\to N, C(1)=p\}
\ee
where 
\be
L_g(C)=\int_0^1 |g(C'(s),C'(s))|^{1/2}\,ds.
\ee
See also Wald and Yip \cite{Wald-Yip}.
This $\tau_{AGH}$ is usually referred to as the \emph{cosmological time function}.   Ebrahimi  \cite{Ebrahimi-2013} has shown that on a Friedman-Robertson-Walker spacetime $\tau_{AGH}$ may be viewed as the time elapsed since the big bang.  

Andersson, Galloway and Howard call this cosmological time {\em regular} if $\tau(p)<\infty$ for all $p\in M$ and $\tau\to 0$ along every inextensible past causal curve.  
While   $\tau_{AGH}$ may not be differentiable, Sormani and Vega  \cite{SV-Null} showed that whenever $\tau_{AGH}$ is regular it is at least \emph{locally anti-Lipschitz} in the sense of Chru\' sciel, Grant and Minguzzi \cite{CGM}. Namely, for every point $p\in N$ there is a neighborhood $U$ of $p$ that has a Riemannian metric with a distance function $d_U: U \times U \to [0,\infty)$ 
such that for all $q,q'\in U$ we have
 \be\label{eqCGM}
q \textrm{ in the causal future of } q' \implies \tau(q)-\tau(q') \geq d_U(q,q') .
 \ee
 
Sormani and Vega  \cite{SV-Null} defined the notion of  \emph{null distance} between two events as the infimum of the
null length over piecewise causal curves:
\be \label{defn-dhat}
\hat{d}_{\tau}(p,q) := \inf \{ \hat{L}_\tau(\beta) : \beta \textrm{ piecewise casual from } p \textrm{ to } q \textrm{ via } x_i \in \beta\}
\ee
so that either $x_i$ is in the causal future of $ x_{i+1}$ or $x_{i+1}$ is in the causal future of $x_i$,
where the \emph{null length} of the curve $\beta$ as in  \eqref{defn-dhat} is
\be \label{defn-Lhat}
\hat{L}_\tau(\beta) := \sum_{i=1}^k |\tau(x_i) - \tau(x_{i-1})|.
\ee
They observe that
\be
\label{dhat-time}
 \dhat_ \tau (p,q) \ge |\tau(p)-\tau(q)| \qquad \text{for any}\, p,q \in N.
\ee
Note that the time function, $\tau$, here need only be a {\emph generalized time function}: $\tau$ increases along causal curves but is not necessarily continuous.

Sormani and Vega  \cite{SV-Null} showed that the null distance converts the Minkowski space endowed with its standard time function into a metric space whose $\hat{d}_\tau$ ball about a point $p$ of radius $R$ is
\emph{causal cylinder}, whose top is the level set $\tau^{-1}(\tau(p)+R)$ 
intersected with the point's \emph{causal future}, $J^+(p)$,  as depicted in Figure
~\ref{fig:cone-cyl}.   

\begin{figure}[h] 
   \centering
   \includegraphics[width=3cm]{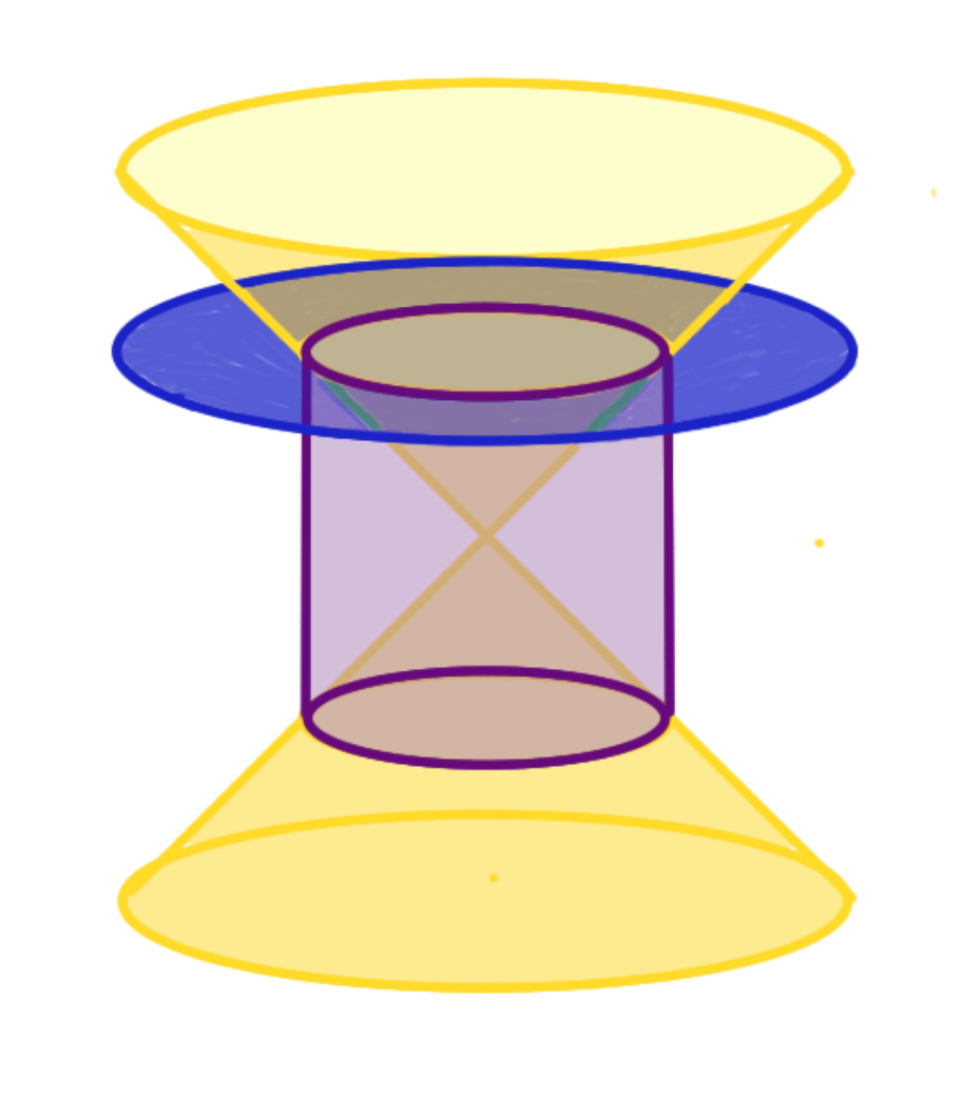} 
   \caption{In Minkowski  space with $\tau=t$ , the $\hat{d}_\tau$ ball is a cylinder 
   aligned with a light cone and a level set of $\tau$, so that (\ref{defn-encodes-causality}) holds.
   }
   \label{fig:cone-cyl}
\end{figure}

Sormani and Vega \cite{SV-Null}  proved that if $\tau$ is a regular cosmological time function, 
$\tau=\tau_{AGH}$,  or, more generally, if $\tau$ is any other time function satisfying the 
anti-Lipschitz condition of Chru\' sciel, Grant and Minguzzi \eqref{eqCGM},
  then $\hat{d}_\tau$ is definite:
\be
q=q' \iff \hat{d}_\tau(q,q')=0
\ee
and induces the topology of the original manifold, $N$.
Sormani and Vega conjectured that under these hypotheses, $\hat{d}_\tau$ also \emph{encodes causality}: 
\be \label{defn-encodes-causality}
\hat{d}_\tau(q,q') = \tau(q)-\tau(q') \iff q \textrm{ is in the causal future of } q'
\ee
and proved this for warped product spacetimes.
For general spacetimes and time functions, it is immediate from the definition of $\hat{d}_\tau$ that 
\be 
q \textrm{ is in the causal future of } q' \quad \implies \quad \hat{d}_\tau(q,q') = \tau(q)-\tau(q')
\ee
but the other direction was shown to be false without a stronger assumption on $\tau$ \cite{SV-Null}.  See Examples~\ref{t-cubed} and~\ref{ex-missing-ray} within.  
 The null distance has been studied further in the work of 
Allen and Burtscher \cite{Allen-Burtscher-20}, Vega \cite{Vega21}, Kunzinger and Steinbauer \cite{Kunzinger-Steinbauer-21},
and Graf and Sormani \cite{Graf-Sormani}.

In this paper we prove that the null distance $\hat{d}_\tau$ locally encodes causality whenever 
$\tau$ satisfies the anti-Lipschitz condition of Chru\' sciel, Grant and Minguzzi.  

\begin{thm} \label{encodes-causality-locally}
Let $(N^{n+1},g)$ be a Lorentzian manifold of dimension $n+1$, $n\geq 1$. Suppose $\tau:N \to [0,\infty)$ is a 
generalized time function that is locally anti-Lipschitz:  about every point $p\in N$ there is a neighborhood $U$ that has a Riemannian metric with a distance function $d_U: U \times U \to [0,\infty)$ 
such that for all $q,q'\in U$ we have:
\be
q \textrm{ in the causal future of } q' \implies \tau(q)-\tau(q') \geq d_U(q,q') .
\ee
Then 
$\hat{d}_\tau$ locally encodes causality: about every point $p\in N$ there is a neighborhood $W$ such that
for all $q \in W$ we have
\be
\hat{d}_\tau(p,q) = \tau(q)-\tau(p) \iff q \textrm{ is in the causal future of } p
\ee
and 
\be
\hat{d}_\tau(p,q) = \tau(p)-\tau(q) \iff q \textrm{ is in the causal past of } p.
\ee
\end{thm}

In particular, we have the following important corollary.

\begin{cor} \label{cor-encodes-causality}
If $\tau$ is a regular cosmological time function then
$\hat{d}_\tau$ locally encodes causality.
\end{cor}

Theorem~\ref{encodes-causality-locally} is proven in Section~\ref{secLocalCausality}.   An outline of the proof is provided at the beginning of that section.  Note that there are examples of spacetimes where the null distance defined using cosmological time does not encode causality globally.  See Example~\ref{ex-missing-ray} within.  Nevertheless we prove global causality in Theorem~\ref{proper-causality} under the additional hypothesis that the time function is proper,  see Section~\ref{sect:proper}.

Once we know the causal structure on a Lorentzian manifold, we can recover the null structure, determine the null cones, and thus also recover the Lorentzian metric up to its conformal class, cf.  \cite[Section 1.9]{Minguzzi-2019}.   That is, if $F: N_1\to N_2$ is a smooth map that preserves causality, 
\be
p \textrm{ is in the causal future of } q \iff F(p) \textrm{ is in the causal future of }F(q), 
\ee 
one can rescale to see that $F^*$ preserves future causal vectors,
\be
\{v\,|\, F_*g_1(v,v)\ge 0\}=\{v\,|\, g_2(v,v)\ge 0\}.
\ee
One can then deduce that $F_*g_1=\phi^2 g_2$ where
$\phi:N_2\to (0,\infty)$ is a conformal factor.   If we also have a pair of smooth cosmological time functions $\tau_i: N_i\to [0,\infty)$ and we know $\tau_1=\tau_2\circ F$, then the fact that
\be
g_i(\nabla \tau_i, \nabla \tau_i)=-1 \implies \phi=1 \textrm{ everywhere. }
\ee
Thus $F: N_1\to N_2$ is a Lorentzian isometry.     

In this paper we prove the following theorem without assuming that $F$ and $\tau_i$ are smooth and using only that $\hat{d}_\tau$ 
encodes causality for a regular cosmological time, $\tau=\tau_{AGH}$. 

\begin{thm} \label{Lorentzian-isom}
Let $(N_1, g_1, \tau_1)$ and $(N_2, g_2, \tau_2)$ be two $n+1$-dimensional Lorentzian manifolds, $n\geq 2$,  equipped with regular cosmological time functions $\tau_i$, $i=1,2$, such that  $\hat{d}_{\tau_i}$ 
encodes causality (for example, this will be the case if $\tau_i$ are proper). If there exists a bijection $F: M_1 \to M_2$ that preserves null distances,
\be
\hat{d}_{\tau_1} (p,q) = \hat{d}_{\tau_2} (F(p), F(q)) \quad  \text{for any} \quad p, q \in N_1,
\ee
and cosmological times,
\be
\tau_1 = \tau_2 \circ F,
\ee 
then $F$ is a diffeomorphism and is a Lorentzian isometry, $F^* g_2 =  g_1$.
\end{thm}

This theorem is proven in Section~\ref{secIsom}. It is natural to ask why we wish to prove Theorem~\ref{Lorentzian-isom} without assuming that $F: M_1\to M_2$ is differentiable.   The short answer is that $F$ is a map between metric spaces and differentiability is not defined between such spaces.   The more serious answer is that we plan to apply this theorem to study limits of sequences of Lorentzian manifolds.   Suppose $(N_j, h_j) \to (N_\infty, h_\infty)$ is defined by requiring that 
\be
(N_j, \hat{d}_{\tau_j}, \tau_j) \to (N_\infty, \hat{d}_{\tau_\infty}, \tau_\infty)
\ee
in the Intrinsic Flat or Gromov-Hausdorff sense with control on the $\tau_j$ as well (to appear in \cite{future-work}), then we only expect the limit spaces to be defined uniquely up to a distance preserving and cosmological time preserving bijection.   In order to guarantee that the limit space is unique when it happens to be a smooth Lorentzian manifold, we need Theorem~\ref{Lorentzian-isom} to be proven without assuming that $F$ is a smooth map.   The same idea arises when studying Gromov-Hausdorff and Intrinsic Flat limits of Riemannian manifolds, $(N, g)$, viewed as metric spaces, $(N, d_g)$, and in that setting one shows the limit spaces are unique up to distance preserving bijection and that a distance preserving bijection between smooth Riemannian manifolds is in fact a smooth Riemannian isometry.

It would be interesting to explore to what extent these theorems hold on lower regularity spacetimes like those studied by Harris \cite{Harris-82}, Alexander and Bishop \cite{Alexander-Bishop-08},
Chru\'sciel and Grant \cite{Chrusciel-Grant-12}, Burtscher \cite{Burtscher-12},
 Kunzinger and S\"amann \cite{Kunzinger-Saemann-18},  Alexander, Graf, Kunzinger and Sämann \cite{AGKS}, Graf and Ling \cite{graf2018maximizers},  Cavalletti and Mondino \cite{Cavalletti-Mondino}, McCann and S{\"a}mann \cite{mccann2021lorentzian}, Ak\'e Hau, Cabrera Pacheco and Solis \cite{AkeHau-etal},
and Burtscher and Garc{\'i}a-Heveling \cite{Burtscher-Garcia-Heveling}.    Note that Kunzinger and Steinbauer have already extended the notion of the null distance to their notion of a Lorentz length space \cite{Kunzinger-Steinbauer-21}.  Within, most of our proofs do not require much regularity as long as piecewise causal curves behave well enough.  However we do apply results of Temple \cite{Temple-1938}, Levichev \cite{Levichev-1987}, Hawking \cite{Hawking-2014}, and Zeeman \cite{Zeeman-64} that would need to be extended. See Remark~\ref{extend-to-low-regularity-spaces}.

\vspace{.2cm}
\noindent{\bf Acknowledgements:}   The authors would like to thank MSRI for its special semester on Mathematical General Relativity where we began collaborating together.   

\section{Examples}

Here we present two examples.   The first one is from \cite{SV-Null} which we repeat here because it clarifies why our proofs involve the reverse Lipschitz property:

\begin{example} \label{t-cubed} 
Consider $\tau = t^3$ on Minkowski space which fails to satisfy the reverse Lipschitz property.    In \cite{SV-Null} it was proven that for any two points $p, q$ in the $\{t = 0\}$ slice, we have $\hat{d}_\tau(p,q) = 0$.    Thus $\hat{d}_\tau$
is both not definite and fails to encode causality both locally and globally.

To see why $\hat{d}_\tau(p,q) = 0$, let $c(s)=(0,x(s))$ be the straight line from $c(0)=p$ to $c(D)=q$ parametrized by Euclidean arclength.   Let 
$\beta_j(s)$
be a piecewise null curve from $\beta_j(0)=p$ to $\beta_j(D)=q$ with $2j$ segments such that 
\begin{eqnarray}
\beta_j(iD/(2j))&=&(0, x(iD/(2j))) \textrm{ for } i \textrm{ even and } \\
\beta_j(iD/(2j))&=&(D/(2j), x(iD/(2j))) \textrm{ for } i \textrm{ odd.}
\end{eqnarray}
Then
\begin{eqnarray}
\hat{d}_\tau(p,q)&\le& \sum_{i=1}^{2j} |\tau(\beta_j(iD/(2j)))-\tau(\beta_j((i-1)D/(2j)))|\\
&=&\sum_{i=1}^{2j} (D/(2j))^3=(2j) (D/(2j))^3\to 0.
\end{eqnarray}
\end{example}

\begin{example} \label{ex-missing-ray} 
Let $N$ be Minkowski upper-half space with a half-line removed:
\be
N= \{(t,x)\, : \, t\in (0,\infty), \, x\in {\mathbb{R}}^3\} \setminus \{ (t,0,0,0)\,:\, t\in [2,\infty)\}
\ee
endowed with the Minkowski metric.   It is easy to see that the cosmological time $\tau=\tau_{AGH}=t$ satisfies the reverse Lipschitz condition.   However,  $\hat{d}_\tau$ does not encode causality.  
To see this, consider the point $p=(1,-1,0,0)$ and $q=(3,1,0,0)$.   In Minkowski space, these points are connected by a future causal curve $C(s)=(1+s,-1+s,0,0)$ which runs from $C(0)=p$ to $C(2)=q$, however this curve runs through 
$C(1)=(2,0,0,0)\notin N$.   In fact, $q$ is not in the causal future of $p$ in $N$.   

Nevertheless, for every $\epsilon>0$ we have a piecewise causal curve which runs
first past causal from $p$ to $p_\epsilon=(1-\epsilon, -1,\epsilon, 0)$ and then along a future causal curve 
\be
C(s)=(1-\epsilon+s,-1+s,\epsilon,0) \textrm{ from } C(0)=p_\epsilon \textrm{ to } C(2)=q_\epsilon
\ee
and then future causal from $q_\epsilon=(3-\epsilon,1,\epsilon,0)$ to $q$.   See Figure~\ref{fig:missing-ray}. Thus
\begin{eqnarray}
\hat{d}_\tau(p,q) &\le& |\tau(p)-\tau(p_\epsilon)|+|\tau(p_\epsilon)-\tau(q_\epsilon)|+|\tau(q)-\tau(q_\epsilon)|\\
&=& \epsilon + 2 +\epsilon \to 2=\tau(q)-\tau(p) \textrm{ as } \epsilon \to 0.
\end{eqnarray}
Since $2=\tau(q)-\tau(p)\ge \hat{d}_\tau(p,q)$ for all $p,q\in M$,  and since the reverse inequality holds true (see \eqref{dhat-time}) we have 
\be\label{missing-encodes-causality}
\hat{d}_\tau(p,q) =\tau(q)-\tau(p) 
\ee
so $\hat{d}_\tau$ does not encode causality globally.   
It does encode causality locally if we take our neighborhoods 
to be small cylindrical blocks that avoid the missing halfline and are thus isometric to cylindrical blocks in Minkowski space.

\end{example}

\begin{figure}[h] 
   \centering
   \includegraphics[width=3cm]{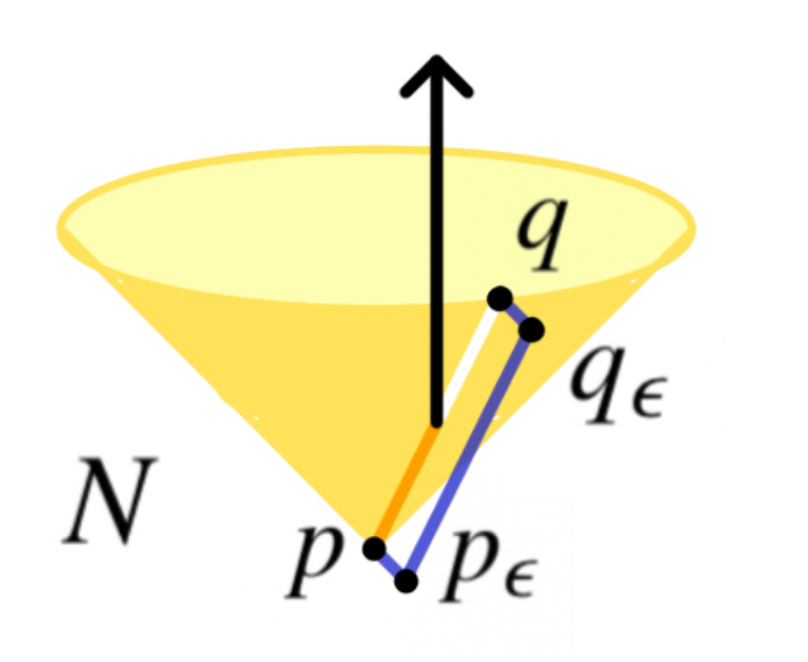} 
   \caption{In Example~\ref{ex-missing-ray}, $N$ is Minkowski upper-half space with a half-line (depicted as a black arrow) removed.  Here $\hat{d}_\tau$ for $\tau=\tau_{AGH}$ fails to encode causality globally: there exist $ p,q$ satisfying (\ref{missing-encodes-causality}) such that
 $q  \notin J^+_p$ because $q$ lies in the shadow of the halfline for light rays from $p$.}
   \label{fig:missing-ray}
\end{figure}

\section{The Null Distance Encodes Causality Locally}\label{secLocalCausality}

In this section we prove Theorem~\ref{encodes-causality-locally}.  In particular, we would like to show that around every point $p$ there is a neighborhood $U$,  such that if $q\in U$,  then $\hat{d}_{\tau}(p,q) = \tau(q)-\tau(p)$ implies $q\in J_+(p)$.

In the first part of Section \ref{secRevLipTime} we do not yet restrict to neighborhoods.  We show that if $\hat{d}_{\tau}(p,q) = \tau(q)-\tau(p)$ holds, then for every $\epsilon>0$ there is a curve $\beta$ from $p$ to $q$ as in \eqref{defn-dhat}, zigzaging backwards and forwards in time and such that the null length of its past directed part is less than $\epsilon$, see Figure~\ref{fig:zig-zag} and Lemma \ref{small-zags-1}. Subsequently, in Lemma \ref{small-zags-2} we prove a refined localized version of this general property under the assumption  that the time function is locally anti-Lipschitz.  However,  note that the existence of a curve $\beta$ as described above does not immediately imply that $q\in J_+(p)$,  even though $\epsilon$ can be chosen to be arbitrarily small.   In fact, we just presented a counter example above, see Example~\ref{ex-missing-ray}.

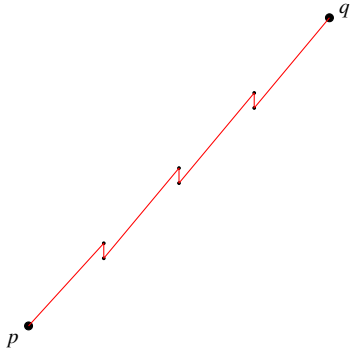
\begin{figure}[h!] \label{fig:zig-zag}
\centering
\begin{tikzpicture}

\node (p) at (-0.2,-0.2){\scriptsize \textit{p}};
\node (q) at (4.2,4.2){\scriptsize \textit{q}};


\filldraw[black](0,0) circle (1.5pt); 
\filldraw[black](1,1.1) circle (0.5pt);
\filldraw[black](1,0.9) circle (0.5pt);
\filldraw[black](2,2.1) circle (0.5pt);
\filldraw[black](2,1.9) circle (0.5pt);
\filldraw[black](3,3.1) circle (0.5pt);
\filldraw[black](3,2.9) circle (0.5pt);
\filldraw[black](4,4.1) circle (1.5pt);

\draw[-,thin,red](0,0) to (1,1.1);
\draw[-,thin,red](1,1.1) to (1,0.9);
\draw[-,thin,red](1,0.9) to (2,2.1);
\draw[-,thin,red](2,2.1) to (2,1.9);
\draw[-,thin,red](2,1.9) to (3,3.1) ;
\draw[-,thin,red](3,3.1)  to (3,2.9);
\draw[-,thin,red](3,2.9) to (4,4.1);

\end{tikzpicture}

\caption{A piecewise causal curve $\beta$ from $p$ to $q$  that almost achieves
$\hat{d}_\tau(p,q)=\tau(q)-\tau(p)$ has very short past causal segments.}
\end{figure}

In order to prove Theorem~\ref{encodes-causality-locally},  we need a suitable ``indicator function" of the causal future,
\be
J_+(p)=\{q\,|\, q \textrm{ is in the causal future of } p\}
\ee 
at least within a neighborhood.  As discussed in Section \ref{secOptical}, it turns out that certain optical functions defined by Temple in 1938 \cite{Temple-1938} are well suited for this purpose. In particular, their level sets are null hypersurfaces generated by null geodesics which can be used to set up a coordinate system capturing the local causal structure of the spacetime, see Theorem \ref{thm-opt}. 

In Section~\ref{secFinalLocalCausality} we apply this coordinate system, combined with the Lipschitzness of the optical function and the anti-Lipschitzness of the time function to complete the proof of Theorem~\ref{encodes-causality-locally}.

\subsection{Almost Minimizing Piecewise Causal Geodesics}\label{secRevLipTime}

Before establishing an implication of  $\hat{d}_{\tau}(p,q) = \tau(q)-\tau(p)$ for ``almost minimizers" for the infimum in the definition of the null distance (\ref{defn-dhat}), we review a basic fact about lengths of piecewise causal curves.   Recall that on any metric space $(M,d)$, a curve is \emph{$d$-rectifiable} if its \emph{$d$-rectifiable length} is finite:
\be \label{rect-length}
L_d(C[a,b])=\sup  \sum_{i=1}^m d(C(s_{i}),C(s_{i-1})) 
\ee
where the supremum is taken over all partitions $a=s_0\le s_1\le \cdots \le s_m=b$.   

\begin{lem}\label{lengths-agree}
On any Lorentzian manifold with any time function $\tau$, and for any piecewise causal curve, $C: [a,b] \to N$,
which is causal on segments $[a_{i-1},a_i]$ where $a=a_0\le a_1\le...\le a_N=b$ we see that the $\hat{d}_\tau$-rectifiable length of the curve agrees with the null length:
\be
L_{\hat{d}_\tau}(C[a,b])=\hat{L}_\tau(C[a,b]) = \sum_{i=1}^N |\tau(C(a_{i}))-\tau(C(a_{i-1}))| =\int_a^b |d/ds(\tau\circ C)|\, ds.
\ee
So the curve can always be reparametrized proportional to $\hat{d}_\tau$-length to
\be
C:[0,1]\to N \textrm{ such that } |d/ds(\tau\circ C)|=\hat{L}_\tau(C[a,b]).
\ee
\end{lem}

\begin{proof}
Given any partition, $a=s_0\le s_1\le \cdots \le s_m=b$, we can take a subpartition 
$a=s'_0\le s'_1\le \cdots \le s'_{m}=b$ such that 
\be
\{s'_0,s'_1,...,s'_{m}\}=\{s_0,s_1,...,s_{m}\} \cup\{a_0, a_1,..., a_N\}
\ee
and by the triangle inequality we have
\be
L_{\hat{d}_\tau}(C[a,b]) \ge \sum_{i=1}^{m'} \hat{d}_\tau(C(s'_{i}),C(s'_{i-1}))\ge \sum_{i=1}^m \hat{d}_\tau(C(s_{i}),C(s_{i-1})).
\ee
Since $C$ is causal on each segment $C[s'_{i-1}, s'_i]$ the middle term is
\be
\sum_{i=1}^{m'} \hat{d}_\tau(C(s'_{i}),C(s'_{i-1}))= \sum_{i=1}^{m'} |\tau(C(s'_{i}))-\tau(C(s'_{i-1}))| = \hat{L}_\tau(C).
\ee
Plugging this back in the middle and taking the supremum over partitions, $a=s_0\le s_1\le \cdots \le s_m=b$,
we see that 
\be
L_{\hat{d}_\tau}(C[a,b]) \ge \hat{L}_\tau(C)\ge \sup \sum_{i=1}^m \hat{d}_\tau(C(s_{i}),C(s_{i-1}))=L_{\hat{d}_\tau}(C[a,b]).
\ee

Let $C$ be a piecewise causal curve as in the formulation of the lemma,  and let $ L= L_{\hat{d}_\tau}(C[a,b])$. Without loss of generality, we may assume that $C:[0,1]\to N$, and that it is causal on the segments $[s_i,  s_{i+1}]$,  $i=0,\ldots, m-1$, where $0=s_0\le s_1\le \cdots \le s_m=1$.  In this case, on each segment $[s_i,  s_{i+1}]$, the function $s\mapsto \tau(C(s))$  is monotone.  If $C$ is future causal on $[s_i,  s_{i+1}]$ we may parametrize it so that
\be
\tau(C(s))=L(s -s_i) + \tau(C( s_i)) \quad \text{ for } s\in [s_i,  s_{i+1}],
\ee
and if $C$ is past causal on $[s_i,  s_{i+1}]$ we may parametrize it so that
\be
\tau(C(s))=-L(s -s_i) + \tau(C( s_i)) \quad \text{ for } s\in [s_i,  s_{i+1}].
\ee
In this case $|d/ds(\tau\circ C)|=L$, and we see that $C$ is  parametrized proportional to 
$\hat{d}_\tau$-rectifiable length because the segments are piecewise causal so within the segments
\be
|\tau (C(s)) - \tau(C(s'))|=\hat{d}_\tau(C(s)),C(s'))=L |s'-s|.
\ee
\end{proof}

Now we consider "almost minimizers" of infimum in the definition of the null distance (\ref{defn-dhat}).

\begin{lem} \label{small-zags-1}
Suppose $\hat{d}_\tau(p,q)= \tau(q)-\tau(p)$ where $\tau$ is a 
time function, then for any $\epsilon>0$
there exists a piecewise causal curve
$\beta:[0,1]\to N$ such that $\beta(0)=p$ and $\beta(1)=q$ 
with 
\be
0=s_0\le s_1\le s_2\le \cdots \le s_{2k}\le s_{2k+1}=1
\ee
such that if $s_{2i} \neq s_{2i+1}$ then 
$\beta$ is future causal on the interval $[s_{2i}, s_{2i+1}]$ for  $i=0, \ldots , k$ and such that if $s_{2i+1} \neq s_{2i+2}$ then  $\beta$ is  past causal on the interval $[s_{2i+1}, s_{2i+2}]$ for $i=0, \ldots, k-1$. Furthermore, we have
\be\label{zig-1}
\sum_{i=0}^{k} (\tau(\beta(s_{2i+1})) - \tau(\beta(s_{2i}))) < \hat{d}_\tau(p,q)+\epsilon
\ee
and 
\be \label{zag-1}
\sum_{i=0}^{k-1} |\tau(\beta(s_{2i+2})) - \tau(\beta(s_{2i+1}))| <\epsilon.
\ee
Moreover, we can parametrize $\beta$ proportional to $\hat{d}_\tau$-arclength (see Lemma~\ref{lengths-agree}).
Thus 
\be \label{zig-param}
\tau(\beta(s))-\tau(\beta(s'))= (s-s') \hat{L}_\tau(\beta) \,\, \text{for any } s,s' \in [s_{2i}, s_{2i+1}], 
\ee
for $i=0, \ldots , k$ and for $i=0, \ldots, k-1$
\be \label{zag-param}
\tau(\beta(s))-\tau(\beta(s'))= -(s-s')  \hat{L}_\tau(\beta) \,\, \text{for any } s,s' \in [s_{2i+1}, s_{2i+2}].
\ee
\end{lem}

Keep in mind that Example~\ref{t-cubed} satisfies the hypotheses of this lemma.

\begin{proof}
By the definition of $\hat{d}_\tau(p,q)$ we know that 
there exists a piecewise causal curve
$\beta:[0,1]\to N$ such that $\beta(0)=p$ and $\beta(1)=q$ 
where $\hat{L}_\tau(\beta) < \hat{d}_\tau(p,q)+\epsilon$.
By allowing $\beta$ to have segments where $s_i = s_{i+1}$, we
can ensure that $\beta$ is as in the formulation of lemma.  In this case we have 
\be \label{zig-1a}
\tau(\beta(s_{2i+1})) - \tau(\beta(s_{2i}))\ge 0
\ee
and
\be \label{zag-1a}
\tau(\beta(s_{2i+2})) - \tau(\beta(s_{2i+1}))\le 0.
\ee
On each interval $[s_i,s_{i+1}]$ where $s_i\neq s_{i+1}$ the function $s\mapsto \tau(\beta(s))$ is monotone. Consequently, we can parametrize $\beta$ starting with the 
first interval so that it satisfies (\ref{zig-param}) and then continuing along the
second interval satisfying (\ref{zag-param}) and so on up and down each interval
until at the end we reach $q=\beta(\hat{L}_\tau(\beta)/\hat{L}_\tau(\beta))=\beta(1)$.

By the
definition of $\hat{L}_\tau(\beta)$ and (\ref{zig-1a}),
we have
\be \label{full-sum}
\hat{L}_\tau(\beta) := \sum_{i=0}^{k-1} |\tau(\beta(s_{2i+2})) - \tau(\beta(s_{2i+1}))|+ \sum_{i=0}^{k} (\tau(\beta(s_{2i+1})) - \tau(\beta(s_{2i})))   
< \hat{d}_\tau(p,q)+\epsilon.
\ee
Dropping the first sum, which is nonnegative, we have (\ref{zig-1}).   

Telescoping our sum and then applying (\ref{zag-1a}), we obtain
\begin{eqnarray*}
\tau(q)-\tau(p)&=&\tau(\beta(s_{2k+1}))-\tau(\beta(s_{0})))\\
&=& \sum_{i=0}^{k-1} (\tau(\beta(s_{2i+2})) - \tau(\beta(s_{2i+1})))+ \sum_{i=0}^{k} (\tau(\beta(s_{2i+1})) - \tau(\beta(s_{2i}))) \\
& \le & \sum_{i=0}^{k} (\tau(\beta(s_{2i+1})) - \tau(\beta(s_{2i}))).
\end{eqnarray*}
Plugging this and the hypothesis of our lemma, $\hat{d}_\tau(p,q)= \tau(q)-\tau(p)$, into
(\ref{full-sum}), we get
\be
\sum_{i=0}^{k-1} |\tau(\beta(s_{2i+2})) - \tau(\beta(s_{2i+1}))|+\tau(q)-\tau(p) < \tau(q)-\tau(p)+\epsilon
\ee
which gives (\ref{zag-1}).
\end{proof}

We now prove a local consequence of this result under the assumption that the time function satisfies the locally anti-Lipschitz condition in the sense of Chru\' sciel, Grant and Minguzzi \cite{CGM}.  Note that the result below holds if we use a different Riemannian metric on $U$, up to rescaling the right hand side.

\begin{lem} \label{small-zags-2}
Given a point $p \in N$ and a neighborhood $U\subseteq N $ about $p$ that has a 
Riemannian metric with a 
distance function $d_U: U \times U \to [0, \infty)$, 
suppose that $\tau$ is a generalized time function satisfying the anti-Lipschitz condition 
(\ref{eqCGM}) for all $q,q'\in U$.   

Let $r_p>0$ be such that $B_{\dhat_\tau}(p, 2r_p) \subset U $
and take any  $\epsilon\in(0, r_p)$.  For any $q\in W_p= B_{\dhat_\tau}(p, r_p) \subset U$ such that
\be
\hat{d}_\tau(p,q)= \tau(q)-\tau(p)
\ee
there exists a piecewise causal curve
$\beta:[0,1]\to U \subseteq N$ such that $\beta(0)=p$ and $\beta(1)=q$ 
where
$\beta$ satisfies all the properties of Lemma~\ref{small-zags-1}
and we have the following estimate for the past causal intervals:
\be\label{zag-s}
\sum_{i=0}^{k-1} |s_{2i+2} - s_{2i+1}| \,<\, \epsilon/\hat{d}_\tau(p,q).
\ee
\end{lem}

Keep in mind that Example~\ref{ex-missing-ray} satisfies the hypotheses of this lemma
where the neighborhood $U$ can be taken to be the entire space.

\begin{proof}
Taking $p,q$ as in the statement, there is a piecewise causal curve $\beta$ defined in Lemma~\ref{small-zags-1}, and parametrized so that (\ref{zig-param}) and (\ref{zag-param}) hold. Consequently, within each causal interval of $\beta$ that lies 
within $U$, we have
\be \label{zigzag}
|s-s'| \,\hat{L}_\tau(\beta) \,=\, |\tau(\beta(s))-\tau(\beta(s'))| \,\ge\, d_U(\beta(s),\beta(s'))
\ee
by the hypothesis (\ref{eqCGM}) and the fact that $\beta(s)$ and $\beta(s')$ are causally related.

We claim that $\beta$ lies entirely in $U$. Indeed, suppose that $\beta$ leaves $U$ so that
\be
s_{max}\,=\,\sup\, \{\,s\,:\, \beta[0,s]\subset U\, \} \,<\, 1.
\ee 
Let $\beta_0$ be the restriction of $\beta$ to the interval  $[0,s_{max}]$. In this case we have $\hat{L}_{\tau}(\beta)\geq \hat{L}_{\tau}(\beta_0)$ which is straightforward to show by introducing an artificial breaking point at $s=s_{max}$, see (\ref{defn-dhat}). Consequently, we have
\be
\hat{L}_{\tau}(\beta)\,\geq \,\hat{L}_{\tau}(\beta_0) \,\geq\, \dhat_{\tau} (\,p,\, \beta(s_{max})\,) \,\geq\, 2 r_p.
\ee
On the other hand, following the proof of Lemma \ref{small-zags-1}, we have chosen $\beta$ so that $\hat{L}_{\tau}(\beta) < \dhat_\tau (p,q) + \epsilon \leq r_p + \epsilon$, a contradiction.

By Lemma~\ref{small-zags-1},  $\beta$ is past causal on intervals $[s_{2i+1}, s_{2i+2}]$ unless $s_{2i+1} =  s_{2i+2}$ for $i=0, \ldots, k-1$.
Since these intervals lie within $U$, we can apply both (\ref{zigzag}) and (\ref{zag-1}), obtaining
\be\label{zig-s}
\sum_{i=0}^{k-1}\,|s_{2i+2} - s_{2i+1}| \,\hat{L}_\tau(\beta)\,\,=\,\,
\sum_{i=0}^{k-1}\, |\tau(\beta(s_{2i+2}))-\tau(\beta(s_{2i+1}))|\,\, <\,\, \epsilon.
\ee
Since $\hat{d}_\tau(p,q) \le \hat{L}_\tau(\beta)$ we have (\ref{zag-s}).
\end{proof}

\subsection{Optical Functions}\label{secOptical}

In general, an optical function on a spacetime $(N,g)$ is a solution $\omega$ of the Eikonal equation $g(\nabla \omega, \nabla \omega)=0$. Its level sets are null hypersurfaces generated by null geodesic segments. Optical functions are important in the study of spacetimes, in particular they are used in the proof of stability of Minkowski spacetime by Christodoulou and Klainerman \cite{Christodoulou-Klainerman-1989}. Many recent results in mathematical general relativity are proven using double null coordinates that are constructed using incoming and outgoing level sets of an optical function. 

In this paper we apply two coordinate systems
introduced in a 1938 paper by Temple  \cite{Temple-1938} that we will call his \emph{future null coordinate chart } and his   \emph{past null coordinate chart }.  The future null coordinate system is depicted in Figure~\ref{fig:Temple-chart}.  

\begin{figure}[h] 
   \centering
   \vspace{.2cm}
   \includegraphics[width=12cm]{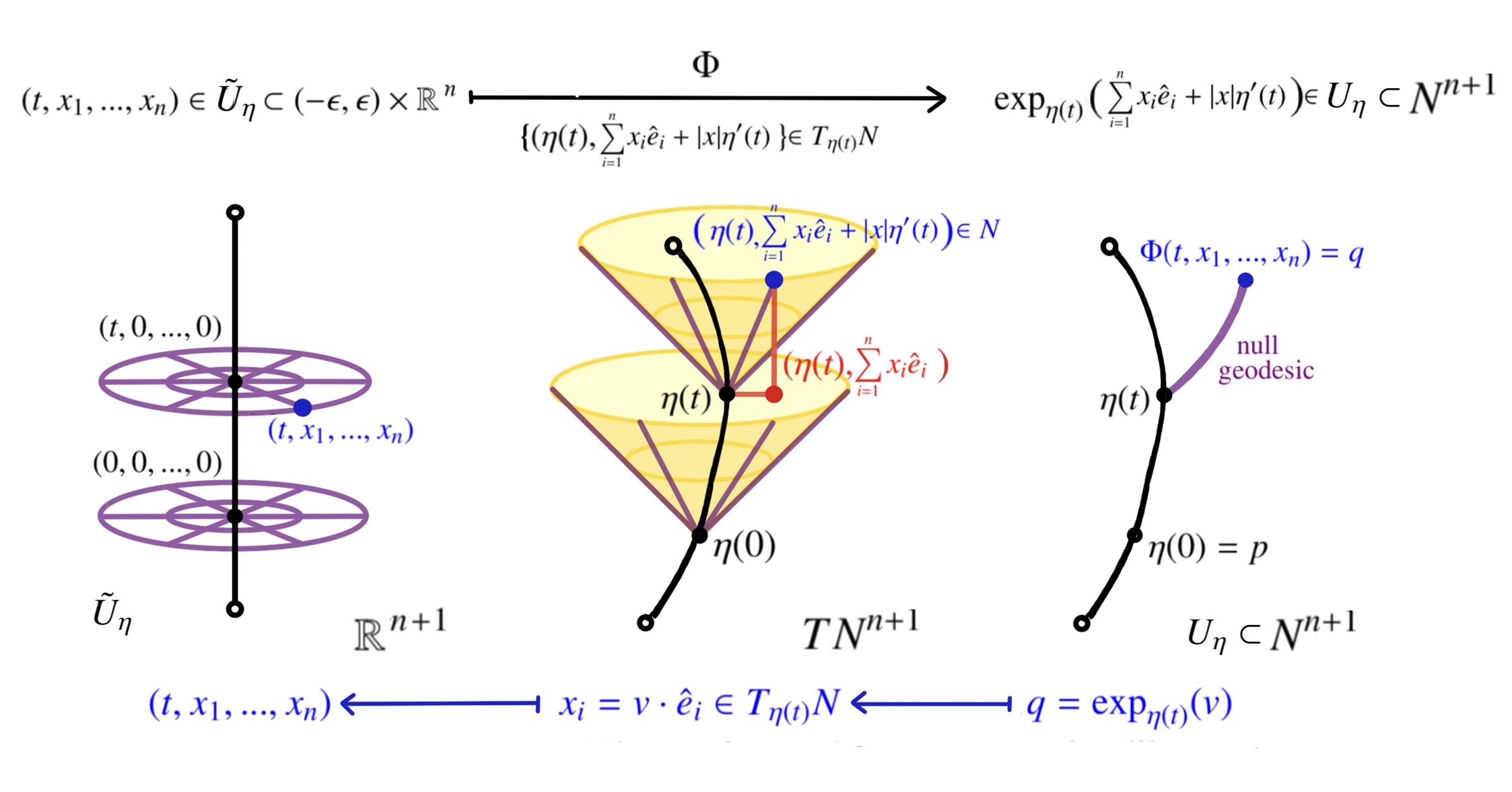} 
   \caption{Temple's future null coordinate chart along a timelike curve.
   }
   \label{fig:Temple-chart}
\end{figure}

\begin{thm}\label{thm-opt} \cite{Temple-1938}
Given any $p$, let $\eta:(-\epsilon, \epsilon) \to N$ be a unit speed future (respectively past) timelike geodesic through $\eta(0)=p$. Let
$\hat{e}_0=\eta'(0)$  and let $\hat{e}_1,...,\hat{e}_n \in T_p N$ be
an orthonormal collection of spacelike vectors such that $\hat{e}_i+\hat{e}_0$ is future (respectively past) null.   We extend this frame
by parallel transport along $\eta$ noting that since $\eta$ is a geodesic, $\eta'(t)=\hat{e}_{0}$ at $\eta(t)$.  Note that for any $x\in{\mathbb R}^m$,
\be
\sum_{i=1}^n x_i \hat{e}_i \,\,+\,\, |x|\, \eta'(t)\,\, \textrm{ is a null vector in } \,\,T_{\eta(t)}N.
\ee
We define a future (respectively past) null chart $\Phi: \tilde{U}_\eta\to U_\eta \subset N$, by
\be
\Phi(t, x_1,...,x_n) = \exp_{\eta(t)}\left(\sum_{i=1}^n x_i \hat{e}_i + |x| \eta'(t)\right)
\ee
which is continuous and invertible on a neighborhood $\tilde{U}_\eta$ of $(-\epsilon, \epsilon) \times \{0\}^n$ and
is smooth away from $\eta$. In this chart, we define the future (respectively past) optical function 
$\omega: U_\eta \to {\mathbb R}$ by
\be
\omega( \Phi(t, x_1,...,x_n) )=t
\ee
and a radial function $\lambda: U_\eta \to {\mathbb R}$ by
\be
\lambda( \Phi(t, x_1,...,x_n) )=\sqrt{x_1^2+\cdots+x_n^2}.
\ee
\end{thm}

Given a unit speed future timelike geodesic $\eta:(-\epsilon, \epsilon) \to N$ through $\eta(0)=p$, we can reverse the parametrization and define both the future and past null charts, 
\be
\Phi_+: \tilde{U}_+ \to U_\eta \,\,\,\textrm{ and }\,\,\,
\Phi_-: \tilde{U}_- \to U_\eta \,\,\,\textrm{ respectively, }
\ee 
together onto the same domain, $U_\eta$, after possibly shrinking their domains.   We let
\be\label{optical-past-future}
\omega_+: U_\eta \to {\mathbb R} \,\,\,\textrm{ and }\,\,\, \omega_-: U_\eta \to {\mathbb R}
\ee
be the future optical function and past optical functions respectively.

\begin{example}
Consider the standard Minkowski spacetime $\mathbb{R}^{n,1}$ with coordinates $(x_0, x_1,\ldots, x_n)$ and the metric $g_{Mink}=-dx_0^2+dx_1^2+\ldots+dx_n^2$. Let $\eta(t)=(t,0,0,0)$, $t\in \mathbb{R}$. Then, in the notations of Theorem \ref{thm-opt}, we have
\be 
\lambda  = r := \sqrt{x_1^2+ \ldots  + x_n^2}
\ee
and the future optical function is given by
\be
\omega_+ (x_0, x_1, \ldots, x_n)=x_0 - r.
\ee
The level sets of $\omega_+$ ,
\be
\omega_+^{-1}(c)=\left\{(x_0, x_1,\ldots, x_n)\,:\, x_0 = c + \sqrt{x_1^2+\ldots + x_n^2} \, \right\} ,
\ee
are the future null cones of $(c,0,\ldots,0)$. 
Note that $\omega_+$ is nonnegative on $J_+((0,0,\ldots,0))$ and negative elsewhere.  
Taking instead $\eta(t)=(-t,0,0,0)$, $t\in \mathbb{R}$,  we find that the past optical function is 
\be
\omega_- (x_0, x_1, \ldots, x_n)=- x_0 - r.
\ee
The level sets of $\omega_-$,
\be
\omega_-^{-1}(c)=\left\{(x_0, x_1,\ldots, x_n)\,:\, x_0 = -c - \sqrt{x_1^2+\ldots + x_n^2} \, \right\} ,
\ee
are the past null cones of $(-c,0,\ldots,0)$.  Clearly,  $\omega_-$ is nonnegative on $J_-((0,0,\ldots,0))$ and negative elsewhere. 
Note that both $\omega_+$ and $\omega_-$ are differentiable away from the $x_0$-axis.  
\end{example}


We will also use the following key property of these charts that we have discovered: 

\begin{lem} \label{omega-q'-q}
Given $p\in N$ let $U_\eta$ be the image of Temple's past and future null coordinate charts centered at $p$.  Let $\omega_+$ be the future optical function. For any $q \in U_\eta$ we have 
\be
\omega_+(q)\ge 0 \Rightarrow q\in J_+(p).
\ee
Furthermore, if  $q, q'\in U_\eta$ then 
\be
q'\in J_+(q) \Rightarrow \omega_+(q') \geq \omega_+(q) .
 \ee
Similarly, if $\omega_-$ is the past optical function then 
\be
\omega_-(q)\ge 0 \Rightarrow q\in J_-(p),
\ee
and 
\be
q'\in J_-(q) \Rightarrow \omega_-(q') \geq \omega_-(q) 
 \ee
 whenever $q, q'\in U_\eta$.
\end{lem}

\begin{proof}
We present the proof in the case when $\omega = \omega_+$ is a future optical function, the necessary modifications in the case of a past optical function $\omega = \omega_-$ are straightforward: one essentially needs to replace "future" by "past" in the arguments below.  Suppose that $\omega(q) \geq 0$. If  $q\in \eta$ then $q=\eta(t)$ for $t=\omega(q)\geq 0$. Since $\eta$ is timelike future directed it follows that $q\in J_+(p)$. If $q\notin \eta$ then there exists a future lightlike geodesic $\gamma_q$ such that $\gamma_q(0)=\eta (t)$ where $t=\omega(q)\geq 0$ and $\gamma_q (1) = q$. Consequently, $q\in J^+(\eta(t))$ and $\eta(t) \in J^+(p)$ which implies that $q\in J^+(p)$. 

Let $\gamma$ be a piecewise smooth future causal geodesic such that $\gamma(0)=q$ and $\gamma(1)= q'$. Without loss of generality we may assume  that $\gamma$ intersects $\eta$ finitely many times. Recalling that $\omega$ is differentiable away from $\eta$ with $g(\nabla \omega,\nabla \omega)=0$, we get 
\begin{eqnarray}
\omega(q')-\omega(q) & =&  \int_0^1  \tfrac{d}{d\sigma} \omega(\gamma(
\sigma))  \, d\sigma \\
&  = & \int_0^1 g( \nabla \omega (\gamma(\sigma)), \gamma'(\sigma)) 
 \, d\sigma \\
&=& \int_0^1 -g( \nabla (-\omega) (\gamma(\sigma)), \gamma'(\sigma)) 
\, d\sigma\\
&\geq& \int_0^1 | \nabla (-\omega) (\gamma(\sigma))|_g | \gamma'(\sigma)|_g
 \, d\sigma\\
&=& 0.
\end{eqnarray}
Here we have used the fact that $\nabla (-\omega)$ is future null and the reverse Cauchy-Schwartz inequality (see e.g. O'Neill \cite[Chapter 5]{O'Neill-text}).
Thus $\omega(q')\geq\omega(q)$ as claimed.
\end{proof}

Using either of Temple's coordinate charts we can define a Riemannian metric on the neighborhood covered by the chart:

\begin{lem} \label{g_R}
Let $\Phi: \tilde{U}_\eta \to U_\eta$ be the future or past null coordinate chart as in Theorem \ref{thm-opt}, with the respective optical function $\omega$. We can define a continuous Riemannian metric on $U_\eta$ by
\be \label{eq-g_R}
g_R(V,W)=\frac{2}{|g(X,X)|} g(X,V)g(X,W) + g(V,W)
\ee
where $X$ is a continuous vector field such that $X=\eta'(t)$ along $\eta$
and $X= \partial_t$ where $t$ is as in Theorem \ref{thm-opt} outside $\eta$. Furthermore, if we define the Riemannian
gradient $\nabla_R \omega$ of the optical function $\omega$ by
\be
g_R(\nabla_R \omega, V)=V(\omega) \quad \text{ for any vector field } V \text{ on } U_\eta
\ee
then
\be\label{riem-grad}
|\nabla_R \omega|_{g_R}=\sqrt{g_R(\nabla_R \omega,\nabla_R \omega)} = \sqrt{2 |g(X,X)|^{-1}}<2
\ee
away from $\eta$, up to shrinking $U_\eta$ if necessary. 

\end{lem}

Note that in this lemma, the Riemannian metric, $g_R^+$, defined by the future null chart does not necessarily agree with the Riemannian metric, $g_R^-$, defined by the past null chart.   We will only use one at a time anyway.

\begin{proof} 
Given a spacetime $(N,g)$ and a continuous timelike vector field $X$, it is straightforward to check using an orthonormal frame 
\be\label{eqONframe}
\left\{e_0=\tfrac{X}{\,\,|g(X,X)|^{1/2}}, e_1, \ldots, e_n\right\} 
\ee
that (\ref{eq-g_R}) defines a continuous Riemannian metric on $N$. The challenge here is to prove that the vector field $X$ as in the formulation of the theorem is continuous on $U_\eta$, even though the optical function $\omega$ is not differentiable along $\eta$ and the vector field $\partial_t$ is a priori not defined on $\eta$.   

We claim that $X$ is continuous along null geodesics emanating
from $\eta$, up to $\eta$.  Fix a unit vector $u\in {\mathbb R}^n$, and consider a family of null geodesics $\{\gamma_{(t,u)}\}_t$ where $\gamma_{(t,u)}$ is a null geodesic starting at $\eta(t)$ and defined for $\lambda \geq 0$ sufficiently small by
\be\label{variation}
\gamma_{(t,u)}(\lambda)=\exp_{\eta(t)}\left(\lambda \left(\sum_{i=1}^n u_i \hat{e}_i +  \eta'(t)\right) \right).
\ee
Note that when $\lambda \neq 0$ this can be written as $\gamma_{(t,u)}(\lambda)=\Phi(t, \lambda u)$. Observe that taking a variation in $t$, we obtain a Jacobi field $J_{(t,u)}(\lambda)$ such that
\be
J_{(t,u)}(0)=\eta'(t)=X \textrm{ and } J_{(t,u)}(\lambda) = \partial_t=X \textrm{ for } \lambda>0.
\ee
Since Jacobi fields are continuous along their geodesic, 
we conclude that $X$ is continuous along $\gamma_{(t,u)}(\lambda)$
including $\lambda=0$.  

We will now apply the fundamental theorem of ordinary differential equations, to show that $X$ is continuous in all directions. Observe that the initial value $J_{(t,u)}(0) = \eta'(t)$ and the initial covariant derivative 
\be
\nabla_{\gamma_{(t,u)}'(0)} J_{(t,u)}(0) = 0
\ee
are smooth along $\eta(t)$, thus the Jacobi fields vary continuously in the $t$ direction for any fixed $u$
and for $\lambda\ge 0$. If we vary $u$ then we are just rotating our initial direction within the null cone, and the variation of the null geodesics is smooth and the Jacobi fields vary continuously as well. 

Finally, we verify \eqref{riem-grad}. Let $\{e_0, e_1, \ldots, e_n\}$ be an orthonormal frame for the Lorentzian metric $g$ on $U_\eta$ as in \eqref{eqONframe}, in particular $e_0$ is timelike, and $e_1, \ldots, e_n$ are spacelike. In what follows, we work on $U_\eta \setminus \eta$. Using \eqref{eq-g_R} it is straightforward to verify that $\{e_0, e_1, \ldots, e_n\}$ is also an orthonormal frame for the Riemannian metric $g_R$, so that 
\be
\nabla_R \omega = e_0 (\omega) e_0 + e_1 (\omega) e_1 + \ldots + e_n (\omega) e_n 
\ee
and 
\be
g_R(\nabla_R \omega, \nabla_R \omega) = (e_0(w))^2 + (e_1(w))^2 + \ldots + (e_n(w))^2.
\ee 
Since $\omega$ satisfies the Eikonal Equation $g(\nabla \omega, \nabla \omega)=0$ away from $\eta$ we have 
\be
(e_0(\omega))^2 = (e_1(\omega))^2 + \ldots + (e_n(\omega))^2,
\ee 
so $g_R(\nabla_R \omega, \nabla_R \omega) = 2 e_0 (\omega)^2$. Recalling that $e_0 =\frac{X}{|g(X,X)|^{1/2}}$ we obtain
\be
g_R(\nabla_R \omega, \nabla_R \omega) =\tfrac{\,\,2X(\omega)^2}{|g(X,X)|} =  \tfrac{2}{|g(X,X)|}
\ee
and \eqref{riem-grad} follows, since by the continuity of $X$  we may assume that $|g(X,X)|>1/2$ in $U_\eta$, as $g(X,X)=-1$ on $\eta$.
\end{proof}

\begin{rmrk}
Note that in the proof above we have only claimed that $X$ is continuous and not smooth.  This is because we never pass through 
$\eta$, and only check that the Jacobi fields are varying continuously for $\lambda \ge 0$.   Even though we could show
the Jacobi fields are differentiably continuous for $\lambda \ge 0$, that does not prove they are differentiable on $\eta$. 
Perhaps they are or perhaps they are not.
\end{rmrk}

\begin{lem} \label{omega-Lip} 
If $g_R$ is the continuous Riemannian metric on $U_\eta$ as defined in Lemma~\ref{g_R} using the future (respectively past) null coordinate chart, and $d_{g_R}$ is the Riemannian distance with respect to $g_R$ defined by 
\be
d_{g_R}(q,q') = \inf\{ L_{g_R}(C)\, :\, C \text{ piecewise smooth curve from } q \text{ to } q'\, \}
\ee
then
\be
\sup\left\{ \frac{|\omega(q)-\omega(q')|}{d_{g_R}(q,q')} \,: \, q\neq q' \in U_\eta \, \right\} < 2.
\ee
where $\omega$ is the future (respectively past) optical function.
\end{lem}

\begin{proof} 
Let $\gamma_i:[0,1]\to U_\eta$  be (piecewise smooth) curves from $\gamma_i(0)=q$ 
to $\gamma_i(1)=q'$ such that
\be
\lim_{i\to \infty} L_{g_R}(\gamma_i) \to d_{g_R}(q,q').
\ee
We can assume that $\gamma_i$ hits the image of $\eta$ at most finitely many 
times, so that
$\omega$ is differentiable along $\gamma_i$ away from those times, with $|\nabla_R \omega|_{
g_R} \leq  2$ by Lemma \ref{g_R}.    Then 
\begin{eqnarray}
|\omega(q)-\omega(q')| & \le &  \int_0^1 | \tfrac{d}{d\sigma} \omega(\gamma_i(
\sigma)) | \, d\sigma \\
&  = & \int_0^1 | g_R( \nabla_R \omega (\gamma_i(\sigma)), \gamma_i'(\sigma)) 
| \, d\sigma \\
&\le & \int_0^1 |\nabla_R \omega (\gamma_i(\sigma))|_{g_R} |\gamma_i'(\sigma))
 |_{g_R} \, d\sigma \\
&\le & 2 \int_0^1 |\gamma_i'(\sigma)) |_{g_R} \, d\sigma \\
&=& 2 L_{g_R}(\gamma_i).
\end{eqnarray}
Taking the limit as $i\to \infty$ we have
\be
\frac{|\omega(q)-\omega(q')|}{d_{g_R}(q,q')} \le 2 \text{ for all } q\neq q' \in U_\eta.
\ee
\end{proof}

\subsection{Proof that the null distance encodes causality locally}\label{secFinalLocalCausality} 

In this section we prove Theorem~\ref{encodes-causality-locally}.

\begin{proof} 
Let $p \in N$.  We will first prove that there is a neighborhood $W$ of $p$ such that for all $q\in W$ we have
\be
\hat{d}_\tau(p,q)=\tau(q)-\tau(p) \implies q \textrm{ is in  the causal future of } p.
\ee
For this,  take any timelike future unit speed geodesic $\eta$ through $p$ and define a neighborhood $U_\eta\subset U$ and the future optical function $\omega: U_\eta\to {\mathbb R}$ as in Theorem~\ref{thm-opt}.
Let $g_R$ be the continuous Riemannian metric on $U_\eta$ as in Lemma~\ref{g_R}. Further, let $r_p$ be such that $B_{\dhat_\tau}(p, 2r_p) \subset U_\eta$ and set $U_p = B_{\dhat_\tau} (p,r_p/2)$. Then for all $p' \in U_p$ we have
\be
U_p= B_{\dhat_\tau} (p,r_p/2) \subset B_{\dhat_\tau} (p',r_p) \subset  B_{\dhat_\tau} (p, 2 r_p) \subset U_\eta.
\ee
Consequently, by choosing $r_{p'} = r_p$ we can ensure that 
\be
\text{ for any } p' \in U_p \textrm{ we have } U_p\subset W_{p'} \textrm{ of Lemma~\ref{small-zags-2}}.
\ee
We will prove that for all $q\in U_p$ we have
\be
\hat{d}_\tau(p,q)=\tau(q)-\tau(p) \implies q \textrm{ is in  the causal future of } p.
\ee

Fix an arbitrary $\epsilon> 0$ and let $\beta:[0,1]\to U \subset M$ such that $\beta(0)=p$ and $\beta(1)=q$ be a piecewise causal curve
as in Lemma ~\ref{small-zags-2}.
By telescoping sums we have
\be
 \omega(q)-\omega(p)
=\sum_{i=0}^{k-1} (\omega(\beta(s_{2i+2})) - \omega(\beta(s_{2i+1})))+ \sum_{i=0}^{k} (\omega(\beta(s_{2i+1})) - \omega(\beta(s_{2i}))) .
\ee
The second sum is nonnegative by Lemma~\ref{omega-q'-q}, as $\beta(s_{2i+1})$ is in the causal future of $\beta(s_{2i})$ as long as $s_{2i}\neq s_{2i+1}$.
Since $\omega(p)=0$ we thereby have
\be
 \omega(q) \ge \sum_{i=0}^{k-1} (\omega(\beta(s_{2i+2})) - \omega(\beta(s_{2i+1}))) 
\ee
where $\beta$ is past causal on $[s_{2i+1},s_{2i+2}]$ if $s_{2i+1}\neq s_{2i+2}$.  Each term in the right hand side is nonpositive but controlled in the view of Lemma~\ref{omega-Lip} and \eqref{zigzag}: 
\begin{equation*}
\begin{split}
\omega(\beta(s_{2i+2})) - \omega(\beta(s_{2i+1})) & \ge - 2 d_{g_R}(\beta(s_{2i+2}), \beta(s_{2i+1})) \\ & \ge - 2 C_p d_U(\beta(s_{2i+2}), \beta(s_{2i+1})) \\ & \ge - 2C_p\hat{L}_\tau(\beta) |s_{2i+2}-s_{2i+1}|
\end{split}
\end{equation*}
where $C_p$ is a constant such that 
\be
C_p^{-1} d_{g_R}(q,z )\leq d_U (q,z)\leq C_p d_{g_R}(q,z) \quad \text{ for all } q,z \in U_p.
\ee  
Finally, applying \eqref{zig-s} we arrive at 
\be
 \omega(q) \ge - 2C_p \hat{L}_\tau(\beta) \sum_{i=0}^{k-1} |s_{2i+2}- s_{2i+1}| \ge - 2 C_p \epsilon. 
\ee
Since this is true for all $\epsilon>0$ we have $\omega(q)\ge 0$ and thus $q$ is in the causal future of $p$.

Finally, we describe  how the above argument needs to be modified to show that there is a neighborhood $W$ of $p$ such that for all $q\in W$ we have
\be
\hat{d}_\tau(p,q)=\tau(p)-\tau(q) \implies q \textrm{ is in  the causal past of } p.
\ee
We let $\eta$ be a timelike past unit speed geodesic through $p$,  $U_\eta$ be the domain of the past null chart and $\omega$ be the past optical function. We define $U_p$ as before, and note that in this case $p \in W_q$ for any $q\in U_p$.  As a consequence, for any $\epsilon> 0$ there is a piecewise causal curve $\beta:[0,1]\to U \subset N$ such that $\beta(0)=q$ and $\beta(1)=p$ with small past segments in the sense of Lemma ~\ref{small-zags-2}.  Equivalently,  there is a piecewise causal curve,  $\beta:[0,1]\to U \subset M$ such that $\beta(0)=p$ and $\beta(1)=q$ with small future segments.  We may now repeat the above argument using the past optical function in the place of future one to conclude that $q\in J_-(p)$. 
\end{proof}

\section{The Null Distance Encodes Causality  when $\tau$ is Proper}\label{sect:proper}

In this section we prove $\hat{d}_\tau$ globally encodes causality for spacetimes with $\tau: N \to (0,T)$ that is proper.   Recall that a function is proper if preimages of compact sets are compact.   So we are assuming that
\be
\tau^{-1}[T_1,T_2] \textrm{ is compact  whenever } 0<T_1\le T_2 <T
\ee
where $T=\sup_N \tau \in (0,\infty]$.   We also assume that $N$ is path connected.

\begin{thm} \label{proper-causality}
Let $(N,g)$ be a path connected spacetime, and let $\tau: N \to \mathbb{R}$ be a generalized proper time function locally satisfying the reverse Lipschitz condition of (\ref{eqCGM}).
Then $\dhat_\tau$ globally encodes causality in $N$, that is 
for all $p,q\in N$,
\be
\hat{d}_\tau(p,q)=\tau(q)-\tau(p) \implies q \in J^+ (p).
\ee
\end{thm}

We can also prove this theorem for more general classes of manifolds but the hypotheses become very technical so we will postpone these theorems for the future.

\begin{rmrk} \label{rmrk-missing-ray}
Recall that in Example~\ref{ex-missing-ray} we gave an example of a path connected spacetime, $(N,g)$,
with a generalized time function, $\tau$, locally satisfying the reverse Lipschitz condition that had a pair of points $p,q$
such that $\hat{d}_\tau(p,q)=\tau(q)-\tau(p)$ but $q$ was not in the causal future of $p$.   The key obstruction
in this example was that a ray was missing which lead to the nonexistence of the causal curve from $p$ to $q$.
In this example the level sets of $\tau$ were not compact, and in particular, the sequence of piecewise
causal curves $\beta_j$ from $p$ to $q$ such that $\hat{L}_\tau(\beta_j)\to \hat{d}_\tau(p,q)$ did not have a converging subsequence.
\end{rmrk}

\subsection{Finding $\hat{d}_\tau$ Minimizing Curves}

In this section we prove that under the appropriate hypotheses on $N$ and $\tau$, every pair of points $p,q\in N$ such that $\hat{d}_\tau (p,q) = \tau(q) - \tau(p)$ has
a curve $C_{p,q}$ whose $\hat{d}_\tau$ rectifiable length achieves $\hat{d}_\tau (p,q)$, see Lemma \ref{lemCpq} below.   This is not true in general as was seen in Example~\ref{ex-missing-ray}.


\begin{lem}\label{lemCpq}
Let $(N,g)$ be a spacetime with a generalized proper time function  $\tau$. Suppose that $p,q\in N$ are such that
\be\label{eqMainAssumption}
\dhat_\tau(p,q)=\tau(q)-\tau(p).
\ee
Then there exists a curve $C_{p,q}= C \subset N$ such that $C(0)=p$, $C(1)=q$ and for all $0\leq s < s' \leq 1$ we have
\be
\dhat_\tau(C(s),C(s'))=\tau(C(s'))-\tau(C(s)).
\ee
\end{lem}

\begin{proof}  
By the definition of null distance, for any $j>0$, there exists a piecewise causal curve $\beta_j \subset N$
from $p$ to $q$ such that 
\be\label{eqBdLengthes}
\dhat_\tau (p,q) \leq L_j=\hat{L}_\tau(\beta_j) < \dhat_\tau(p,q)+1/j.
\ee
Recall that by Lemma 3.6 \cite{SV-Null} we have 
\be\label{eqRecall}
\hat{L}_\tau (\beta_j) \ge \max_{\beta_j} \tau  - \min_{\beta_j} \tau.
\ee
As a consequence of \eqref{eqBdLengthes}, \eqref{eqRecall} and \eqref{eqMainAssumption} we have
\begin{eqnarray*}
\tau(q)-\tau(p) + 1/j & \ge & \max_{\beta_j} \tau  - \tau(p), \\
 \tau(q)-\tau(p) + 1/j  & \ge &  \tau(q)  - \min_{\beta_j} \tau.
\end{eqnarray*}
It follows that 
\be
\tau(p) - 1/j  \leq \min_{\beta_j} \tau <  \max_{\beta_j} \tau \leq \tau(q)+ 1/j .
\ee
Since $\tau(p),\tau(q) \in (0,T)$, by taking $j$ to be sufficiently large we see that all 
$\beta_j$ are contained in a compact set $\tau^{-1}([T_1,T_2])$ for some $0<T_1 \leq T_2 <T$.

Since $\beta_j$ is piecewise causal, by Lemma~\ref{lengths-agree} we can parametrize each $\beta_j$ 
proportional to 
$\hat{d}_\tau$-rectifiable length. In this case, we have $\beta_j:[0,1]\to N$  and 
 for any $s,s' \in [0,1]$ we have 
\be\label{eqApproxLength}
{L}_{\hat{d}_\tau} (\beta_j|_{[s,s']})= L_j |s-s'| \geq |\tau (\beta_j(s)) - \tau( \beta_j(s'))|.
\ee

Since $L_j\to L= \dhat_\tau(p,q)$ and since all the images of $\beta_j$ are contained in the compact set, $\tau^{-1}([T_1,T_2])$, we can apply the Arzela-Ascolli theorem (cf. \cite[Theorem 2.5.14]{BBI})  
to see that there is a subsequence of $\{\beta_j\}$, denoted by the 
same notation, that $C^0$ converges to a curve $C:[0,1] \to N$ that is contained in $\tau^{-1}([T_1,T_2])$.
By lower semicontinuity of $L_{\dhat_ \tau}$ (cf. \cite[Proposition 2.3.4]{BBI}), we have
\be\label{lower-semi-length}
L_{\dhat_ \tau}(C) \leq \liminf_{j\to \infty}L_{\dhat_ \tau}(\beta_j)  =  \liminf_{j\to \infty}\hat{L}_\tau (\beta_j)= \dhat_ \tau (p,q).
\ee
On the other hand, by the triangle inequality and definition of rectifiable length, 
\be
L_{\dhat_ \tau}(C) \geq \dhat_ \tau (p,q). 
\ee
Combining this with our hypothesis we have
\be\label{full-length}
L_{\dhat_ \tau}(C) = \dhat_ \tau (p,q)=\tau(q)-\tau(p).
\ee 

Note also that for any $0\le s<s'\le 1$ we have
\be\label{part-length}
L_{\dhat_ \tau}(C[s,s']) \ge \dhat_ \tau (C(s),C(s')) \ge |\tau(C(s)-\tau(C(s'))| \ge \tau(C(s)-\tau(C(s'))
\ee
by the definition of rectifiable length, triangle inequality, and (\ref{dhat-time}).

We now turn to the proof of the last claim. Assume on the contrary that there is a segment $[s,s'] \subseteq [0,1]$ such that
\be
 \dhat_ \tau (C(s), C(s')) > \tau (C(s)) - \tau(C(s')).
\ee 
Applying (\ref{part-length}) to $[0,s]$, $[s,s']$, and $[s',1]$ we would have
\begin{eqnarray*}
L_{\dhat_ \tau}(C[0,1]) &=&L_{\dhat_ \tau}(C[0,s]) +L_{\dhat_ \tau}(C[s,s']) +L_{\dhat_ \tau}(C[s',1]) \\
&\ge& \dhat_ \tau (C(0),C(s))+\dhat_ \tau (C(s),C(s'))+\dhat_ \tau (C(s'),C(1))\\
& >&  (\tau (C(0)) - \tau(C(s)))+ (\tau (C(s)) - \tau(C(s')))+ (\tau (C(s')) - \tau(C(1))) \\
&=&  \tau (C(0)) - \tau(C(1)).
\end{eqnarray*}
which contradicts (\ref{full-length}).
\end{proof}

As seen from this proof, the curve $C_{p,q}$ such that its $\dhat_\tau$-rectifiable length achieves $\dhat_\tau (p,q)$ will exist if we replace \eqref{eqMainAssumption} and the hypothesis of properness of $\tau$ by any other  assumption that ensures that almost $\dhat_\tau$-minimizing curves $\beta_j$ are contained in a compact set.  Note that in general such a curve $C_{p,q}$ need not be piecewise causal even though it is a limit of piecewise causal curves $\beta_j$ as we see in the following example:

\begin{example} \label{ex-wise}

Consider for instance $p=(0,0,0,0)$ and $q=(0,1,0,0)$ in Minkowski space with $\tau=t$.   Let $\beta_j$ be piecewise causal curves running from $C(0)=p=p_0$ to $C(1)=q=p_{2j}$
via the points
\be
p_{2i}=(0, i/j,0,0) \textrm{ and } p_{2i+1}=(1/(2j) + (1/j^2) , i/j+1/(2j),0,0)
\ee
with 
\begin{eqnarray}
\hat{L}_\tau(\beta_j)&=& \sum_{i=1}^{j}\left( |\tau(p_{2i-1})-\tau(p_{2i-2})|+|\tau(p_{2i})-\tau(p_{2i-1})| \right) \\
&=&  \sum_{i=1}^{j} 2(1/(2j) + (1/j^2)) =j (2(1/(2j) + (1/j^2))) = 1+ 2/j .
\end{eqnarray}
As seen in the Minkowski space example in \cite{SV-Null},  $\hat{d}_\tau(p,q)=1$, so as
 $j \to \infty$, we see
 \be
 \hat{L}_\tau(\beta_j) \to  \hat{d}_\tau(p,q).
 \ee
It is easy to see that these $\beta_j$  converge in the $C^0$ sense to 
\be
C_{p,q}(s)=(0,s,0,0)
\ee
which is spacelike from $p$ to $q$ and not piecewise null.   
\end{example}


\subsection{Local-to-Global with a Proper Time Function}

We can now prove Theorem~\ref{proper-causality}.  

\begin{proof} 
Let $p,q\in N$ be such that
\be
\hat{d}_\tau(p,q)=\tau(q)-\tau(p).
\ee
In this case, Lemma~\ref{lemCpq} implies that
there is a curve $C=C_{p,q} : [0,1]\to N$ such that $C(0)=p$, $C(1)=q$, and for all
$0\leq s < s' \leq L$ we have 
\be
L_{\hat{d}_\tau}(C|_{[s,s']})=\dhat_\tau(C(s),C(s'))=\tau(C(s'))-\tau(C(s)).  
\ee

For each point $C(t)$,  $t\in [0,L]$, we have an open neighborhood $U_{C(t)}$ where $\hat{d}_\tau$ locally encodes causality.   Since $C([0,L])$ is compact,  
we have finitely many $U_i=U_{C(s_i)}$ required to cover it.  Taking
\be
0=t_0<s_1<t_1<\cdots<s_N< t_N=L
\ee
so that 
$
C(t_i), C(s_{i+1}), C(t_{i+1})  \in U_i, 
$
for all $i=0, \ldots, N-1$ we have
\be
\hat{d}_\tau(C(s_{i+1}),C(t_i)) = \tau(C(s_{i+1}))-\tau(C(t_i))
\ee
and for all $i=1, \ldots, N$ we have
\be
\hat{d}_\tau(C(s_{i}),C(t_i)) = \tau(C(t_{i}))-\tau(C(s_i)).
\ee
By local causality on $U_i$,  for all $i=0, \ldots, N-1$ we have
\be
C(s_{i+1}) \textrm{ is in the causal future of } C(t_i),
\ee
and and for all $i=1, \ldots, N$ we have
\be
C(t_{i}) \textrm{ is in the causal future of } C(s_i).
\ee
Thus $q=C(L)$ is in the causal future of $p=C(0)$.
\end{proof}

\section{The Isometry Theorem} \label{secIsom}

In this section we will prove Theorem~\ref{Lorentzian-isom}.   
This theorem concerns a pair of spacetimes, $(N_1,g_1)$ and $(N_2,g_2)$, equipped with regular  cosmological time functions, $\tau_1$ and $\tau_2$, and a bijection $F: N_1 \to N_2$ that preserves null distances,
\be \label{hat-here}
\hat{d}_{\tau_1} (p,q) = \hat{d}_{\tau_2} (F(p), F(q)) \quad  \text{for any} \quad p, q \in N_1,
\ee
and cosmological times,
\be \label{tau-here}
\tau_1 = \tau_2 \circ F.
\ee
We assume that the cosmological time functions $\tau_i$ are proper, so that the causality is encoded by the associated null distances $\hat{d}_{\tau_i}$ by Theorem \ref{proper-causality}.

As explained in the introduction, if $F$ was known to be a diffeomorphism, then the fact that it preserves the causal structure would immediately imply that it is a conformal isometry.  
However we do not know that $F$ is even differentiable.

To prove that $F$ is a conformal isometry, we will apply the following theorem of Levichev \cite{Levichev-1987}, which builds upon a lemma in Hawking's 1966 Adams Prize Essay (that appears as Lemma 19 in the reprint \cite{Hawking-2014}) which in turn builds upon work of Zeeman \cite{Zeeman-64}.
See also the paper by Hawking-King-McCarthy \cite{Hawking-King-McCarthy-1976} and the work of Malament \cite{Malament-1977}). 
See Minguzzi's recent survey \cite{Minguzzi-2019} for an overview of these results.   

\begin{thm}\label{thLevichev} {\em [Levichev]}
Let $(N_1,g_1)$ and $(N_2,g_2)$ be two $n+1$-dimensional distinguishing spacetimes, $n\geq 2$,   and let $F: N_1 \to N_2$ be a causal bijection, i.e. a bijection such that
\be
q\in J_+(p) \iff F(q) \in J_+ (F(q)).
\ee
 Then $F$ is a smooth conformal isometry, i.e. there exists a smooth function $\phi>0$ such that $F^* g_2 = \phi^2 g_1$.
\end{thm}

Anderson, Galloway and Howard \cite{AGH} proved that a regular cosmological time function is continuous, which implies that the spacetimes $(N_i,g_i)$ are distinguishing, see e.g. \cite[Theorem 4.5.8 (v')]{Minguzzi-2019}.   Thus Levichev's Theorem can be applied on manifolds with regular cosmological time functions.  In particular we can apply it in our proof of  Theorem~\ref{Lorentzian-isom}:

\begin{proof}[Proof of Theorem~\ref{Lorentzian-isom}]

The assumptions of the theorem together with Theorem \ref{proper-causality} imply that $F: N_1\to N_2$  is a causal bijection.  Since our spacetimes have regular cosmological time functions and are thus distinguishing, we may apply Theorem \ref{thLevichev}, to conclude that $F$ is a conformal isometry: there exists a  smooth positive function $\phi: N_1 \to \mathbb{R}$ so that 
\be \label{conf-eq}
F^* g_2 = \phi^2 \, g_1 \textrm{ on }N_1.
\ee

We will show that $\phi \equiv 1$, so that $F: N_1 \to N_2$ is an isometry. Assume without loss of generality that there is $q \in N_1$ such that that $\phi(q) > 1$ (the case  $\phi(q) < 1$ is treated similarly). Let $U \subseteq N_1$ be a precompact neighborhood of $q$ such that $\phi>1$ in $U$ and let $V$ be the closure of $U$. We will compute the volume of $V$ in two different ways to reach a contradiction. On the one hand, we have
\be\label{133}
\int_{F(V)} \, d\mu^{g_2} = \int_{V} \, d\mu^{F^* g_2} =\int_{V} \, d\mu^{\phi^2 \, g_1} = \int_{V} \phi^n \, d\mu^{g_1}.
\ee 
Suppose that $\tau_1(V)=[\tau_{\min},\tau_{\max}]$, then we also have $\tau_2(F(V))=[\tau_{\min},\tau_{\max}]$. Since the cosmological time functions $\tau_i$, $i=1,2$, are Lipschitz with $|\nabla \tau_i|_{g_i} = 1$ almost everywhere, we can  compute:
\begin{equation*}
\begin{split}
\int_{F(V)} \, d\mu^{g_2} &= \int_{\tau_{\min}}^{\tau_{\max}} \left(\int_{\tau_2^{-1}(\tau)} d\mathcal{H}_{n-1}^{g_2}\right) d\tau \quad \qquad \textrm{ by the coarea formula }\\
& = \int_{\tau_{\min}}^{\tau_{\max}} \left(\int_{F^{-1}(\tau_2^{-1}(\tau))} d\mathcal{H}_{n-1}^{F^* g_2}\right) d\tau 
\quad  \textrm{ by change of variables,}\\
& = \int_{\tau_{\min}}^{\tau_{\max}} \left(\int_{\tau_1^{-1}(\tau)} d\mathcal{H}_{n-1}^{F^* g_2}\right) d\tau \quad \quad\,\,\,
\textrm{ by $\tau_2\circ F=\tau_1$,}
\\ & = \int_{\tau_{\min}}^{\tau_{\max}} \left(\int_{\tau_1^{-1}(\tau)} \phi^{n-1} d\mathcal{H}_{n-1}^{g_1}\right) d\tau 
\quad \textrm{ by (\ref{conf-eq}) applied in dimension $(n-1)$,}
\\
& = \int_{V}  \phi^{n-1} \, d\mu^{g_1} \qquad \qquad \qquad \,\,\,\quad \textrm{ by the coarea formula again.}
\end{split}
\end{equation*}
If we subtract this from (\ref{133}) we have
\be
\int_{V} \phi^{n-1} (\phi-1) \, d\mu^{g_1} =\int_{V} \phi^{n} \, d\mu^{g_1}-\int_{V} \phi^{n-1} \, d\mu^{g_1}= 0,
\ee
which contradicts $\phi > 1$ in $V$. Consequently, $\phi \equiv 1$ and  $F^* g_2 =  g_1$ on $N$. 
\end{proof} 

Note that in the above proof we strongly use that $\tau_i$ are cosmological time functions.  

\begin{example} \label{only-conformal}
Let $(N_1,g_1)$ be any Lorentzian manifold with any time function $\tau_1$.   If we let
$N_2=N_1$ and $g_1=\phi^2 g_2$ and $\tau_1=\tau_2$ then
\be
\hat{d}_{\tau_1}(p,q)=\hat{d}_{\tau_2}(p,q)
\ee
because the same curves, $\beta$, are piecewise causal with respect to both $g_1$ and $g_2$
and 
\be
\hat{L}_{\tau_1}(\beta) = \sum_{i=1}^k |\tau_1(x_i) - \tau_1(x_{i-1})|
=
 \sum_{i=1}^k |\tau_2(x_i) - \tau_2(x_{i-1})|=\hat{L}_{\tau_2}(\beta).
\ee
Thus the identity map $F: N_1\to N_2$ is an isometry between $(N_1, \hat{d}_{\tau_1})$
and $(N_2, \hat{d}_{\tau_2})$ which preserves the time functions, $\tau_1=\tau_2\circ F$,
but $F$ is not a Lorentzian isometry.
\end{example}

\begin{rmrk}
Note the dimensional restriction in Theorem~\ref{Lorentzian-isom} is a consequence of the respective restriction in Levichev's Theorem \ref{thLevichev} \cite{Levichev-1987} and Hawking's 1966 work that Levichev builds upon.   In Hawking's proof (cf. \cite{Hawking-2014}), the dimension condition is essential for proving the differentiability of $F$.   We do not know of a non-example demonstrating that their dimension condition is necessary to prove that $F$ is differentiable.
Nor do we know of a non-example demonstrating that the dimension condition is necessary to prove our isometry theorem. 
\end{rmrk}

\begin{rmrk} \label{extend-to-low-regularity-spaces}
It should also be noted that we have deliberately proven our theorem using the coarea formula so that we may then imitate this proof on a lower regularity space like the integral current spaces of Sormani-Wenger \cite{SW-JDG} in the future.   Levichev's Theorem is very much proven in the style of Alexandrov geometry and should extend easily.   Hawking's work is harder to dissect. More recent work studying Lorentz spaces of lower regularity might be applied to extend the work of Levichev and Hawking to these settings and thus extend our isometry theorem.   See work of Harris \cite{Harris-82}, Alexander and Bishop \cite{Alexander-Bishop-08},
Chru\'sciel and Grant \cite{Chrusciel-Grant-12}, Burtscher \cite{Burtscher-12},
 Kunzinger and Sämann \cite{Kunzinger-Saemann-18},  Alexander, Graf, Kunzinger and Sämann \cite{AGKS}, Graf and Ling \cite{graf2018maximizers},  Cavalletti-Mondino \cite{Cavalletti-Mondino}, McCann-S{\"a}mann \cite{mccann2021lorentzian}, Ak\'e Hau, Cabrera Pacheco and Solis \cite{AkeHau-etal},
and Burtscher and Garc{\'i}a-Heveling \cite{Burtscher-Garcia-Heveling}.  
\end{rmrk}

\bibliographystyle{plain}
\bibliography{SF-bib.bib}

\begin{thebibliography}{10}

\bibitem{AkeHau-etal}
Luis Ak\'{e}~Hau, Armando~J. Cabrera~Pacheco, and Didier~A. Solis.
\newblock On the causal hierarchy of {L}orentzian length spaces.
\newblock {\em Classical Quantum Gravity}, 37(21):215013, 22, 2020.

\bibitem{Alexander-Bishop-08}
Stephanie~B. Alexander and Richard~L. Bishop.
\newblock Lorentz and semi-{R}iemannian spaces with {A}lexandrov curvature
  bounds.
\newblock {\em Comm. Anal. Geom.}, 16(2):251--282, 2008.

\bibitem{AGKS}
Stephanie~B. Alexander, Melanie Graf, Michael Kunzinger, and Clemens Sämann.
\newblock Generalized cones as {L}orentzian length spaces: {C}ausality,
  curvature, and singularity theorems.
\newblock {\em arXiv:1909.09575}, 2021.

\bibitem{Allen-Burtscher-20}
Brian Allen and Annegret Burtscher.
\newblock Properties of the null distance and spacetime convergence.
\newblock {\em Int. Math. Res. Not. IMRN}, 10:7729--7808, 2022.

\bibitem{AGH}
Lars Andersson, Gregory~J. Galloway, and Ralph Howard.
\newblock The cosmological time function.
\newblock {\em Classical Quantum Gravity}, 15(2):309--322, 1998.

\bibitem{BBI}
Dmitri Burago, Yuri Burago, and Sergei Ivanov.
\newblock {\em A course in metric geometry}, volume~33.
\newblock American Mathematical Society, 2022.

\bibitem{Burtscher-Garcia-Heveling}
Annegret Burtscher and Leonardo Garcia-Heveling.
\newblock Time functions on {L}orentzian length spaces.
\newblock {\em arXiv:2108.02693}, 2021.

\bibitem{Burtscher-12}
Annegret~Y. Burtscher.
\newblock Length structures on manifolds with continuous {R}iemannian metrics.
\newblock {\em New York J. Math.}, 21:273--296, 2015.

\bibitem{Cavalletti-Mondino}
Fabio Cavalletti and Andrea Mondino.
\newblock Optimal transport in {L}orentzian synthetic spaces, synthetic
  timelike {R}icci curvature lower bounds and applications.
\newblock {\em arXiv:2004.08934}, 2020.

\bibitem{Christodoulou-Klainerman-1989}
Demetrios Christodoulou and Sergiu Klainerman.
\newblock The nonlinear stability of the {M}inkowski metric in general
  relativity.
\newblock In {\em Nonlinear Hyperbolic Problems}, pages 128--145. Springer,
  1989.

\bibitem{CGM}
Piotr~T. Chru\'{s}ciel, James D.~E. Grant, and Ettore Minguzzi.
\newblock On differentiability of volume time functions.
\newblock {\em Ann. Henri Poincar\'{e}}, 17(10):2801--2824, 2016.

\bibitem{Chrusciel-Grant-12}
Piotr~T. Chru{\'s}ciel and James~D.E. Grant.
\newblock On {L}orentzian causality with continuous metrics.
\newblock {\em Classical Quantum Gravity}, 29(14):145001, 2012.

\bibitem{Ebrahimi-2013}
Neda Ebrahimi.
\newblock Some observations on cosmological time functions.
\newblock {\em J. Math. Phys.}, 54(5):052503, 3, 2013.

\bibitem{graf2018maximizers}
Melanie Graf and Eric Ling.
\newblock Maximizers in {L}ipschitz spacetimes are either timelike or null.
\newblock {\em Classical Quantum Gravity}, 35(8):087001, 2018.

\bibitem{Graf-Sormani}
Melanie Graf and Christina Sormani.
\newblock Lorentzian area and volume estimates for integral mean curvature
  bounds.
\newblock {\em arXiv:2106.02319}, 2021.

\bibitem{Harris-82}
Steven~G. Harris.
\newblock A triangle comparison theorem for {L}orentz manifolds.
\newblock {\em Indiana Univ. Math. J.}, 31(3):289--308, 1982.

\bibitem{Hawking-2014}
Stephen~W. Hawking.
\newblock Singularities and the geometry of spacetime: reprint of the {A}dams
  {P}rize essay 1966.
\newblock {\em Eur Phys J H}, 39:413–503, 2014.

\bibitem{Hawking-King-McCarthy-1976}
Stephen~W. Hawking, Andrew~R. King, and Patrick~J. McCarthy.
\newblock A new topology for curved space -- time which incorporates the
  causal, differential, and conformal structures.
\newblock {\em Journal of Mathematical Physics}, 17(2):174--181, 1976.

\bibitem{Kunzinger-Saemann-18}
Michael Kunzinger and Clemens S\"{a}mann.
\newblock Lorentzian length spaces.
\newblock {\em Ann. Global Anal. Geom.}, 54(3):399--447, 2018.

\bibitem{Kunzinger-Steinbauer-21}
Michael Kunzinger and Roland Steinbauer.
\newblock Null distance and convergence of {L}orentzian length spaces.
\newblock {\em Ann. Henri Poincar\'{e}}, page 1—24, 2022.

\bibitem{Levichev-1987}
Aexander~V. Levichev.
\newblock The causal structure of a {L}orentzian manifold determines its
  conformal geometry.
\newblock {\em Dokl. Akad. Nauk SSSR}, 293(6):1301--1305, 1987.

\bibitem{Malament-1977}
David~B. Malament.
\newblock The class of continuous timelike curves determines the topology of
  spacetime.
\newblock {\em J. Mathematical Phys.}, 18(7):1399--1404, 1977.

\bibitem{mccann2021lorentzian}
Robert~J McCann and Clemens S{\"a}mann.
\newblock A {L}orentzian analog for {H}ausdorff dimension and measure.
\newblock {\em arXiv preprint arXiv:2110.04386}, 2021.

\bibitem{Minguzzi-2019}
Ettore Minguzzi.
\newblock Lorentzian causality theory.
\newblock {\em Living reviews in relativity}, 22(1):1--202, 2019.

\bibitem{O'Neill-text}
Barrett O'Neill.
\newblock {\em Semi-{R}iemannian geometry with applications to relativity}.
\newblock Academic press, 1983.

\bibitem{future-work}
Anna Sakovich and Christina Sormani.
\newblock Future work.

\bibitem{Sormani-Oberwolfach-18}
Christina Sormani.
\newblock Oberwolfach report 2018: {S}pacetime intrinsic flat convergence.
\newblock {\em Oberwolfach Reports}, 2018.

\bibitem{SV-Null}
Christina Sormani and Carlos Vega.
\newblock Null distance on a spacetime.
\newblock {\em Classical Quantum Gravity}, 33(8):085001, 29, 2016.

\bibitem{SW-JDG}
Christina Sormani and Stefan Wenger.
\newblock The intrinsic flat distance between {R}iemannian manifolds and other
  integral current spaces.
\newblock {\em Journal of Differential Geometry}, 87(1):117--199, 2011.

\bibitem{Temple-1938}
George Temple.
\newblock New systems of normal co-ordinates for relativistic optics.
\newblock {\em Proceedings of the Royal Society of London. Series A.
  Mathematical and Physical Sciences}, 168(932):122--148, 1938.

\bibitem{Vega21}
Carlos Vega.
\newblock Spacetime distances: an exploration.
\newblock {\em arXiv:2103.01191}, 2021.

\bibitem{Wald-Yip}
Robert~M. Wald and Ping Yip.
\newblock On the existence of simultaneous synchronous coordinates in
  spacetimes with spacelike singularities.
\newblock {\em J. Math. Phys.}, 22:2659–2665, 1981.

\bibitem{Zeeman-64}
Erik~C. Zeeman.
\newblock Causality implies the {L}orentz group.
\newblock {\em J. Math. Phys. 5}, 490, 1964.

\end{thebibliography}

\end{document}